\documentclass[11pt,a4paper]{amsart}

\usepackage{style}
\usepackage[left=3.4cm,right=3.4cm,bottom=3.4cm,top=3.4cm]{geometry}

\usepackage{times}
\usepackage{diagbox}

\author{Gebhard Martin}
\address{Mathematisches Institut \\ Universität Bonn \\ Endenicher Allee 60 \\ 53115 Bonn \\ Germany}
\email{gmartin@math.uni-bonn.de} 

\author{Réka Wagener}
\address{}
\email{rekawagener@t-online.de}

\title{Classification of non-$F$-split del Pezzo surfaces of degree $1$}

\date{\today}
\subjclass[2020]{}

\begin{document}

\begin{abstract}
Using Fedder's criterion, we classify all non-$F$-split del Pezzo surfaces of degree $1$. We give a necessary and sufficient criterion for the $F$-splitting of such del Pezzo surfaces in terms of their anti-canonical system.
\end{abstract}

\maketitle

\tableofcontents

\section{Introduction}
Let $k$ be an algebraically closed field of characteristic $p > 0$. Recall that a \emph{del Pezzo surface} over $k$ is a smooth projective surface $X$ over $k$ with ample anti-canonical sheaf $\omega_X^{-1}$. The \emph{degree} of $X$ is defined as $d \coloneqq K_X^2$ and it satisfies $1 \leq d \leq 9$. A scheme $S$ over $k$ is called \emph{$F$-split} if the $k$-linear Frobenius $F: \mathcal{O}_S \to F_* \mathcal{O}_S$ splits as a morphism of $\mathcal{O}_S$-modules.
The following is known about non-$F$-split del Pezzo surfaces:

\begin{theorem} \label{thm: known}
Let $X$ be a del Pezzo surface of degree $d$ over an algebraically closed field $k$ of characteristic $p > 0$. Then, the following hold:
\begin{enumerate}
    \item \cite[Example 5.5]{Hara} If $X$ is not $F$-split, then the tuple $(d,p)$ satisfies one of the following:
    \begin{enumerate}
    \item $d = 3$, $p = 2$
    \item $d = 2$, $p \in \{2,3\}$.
    \item $d = 1$, $p \in \{2,3,5\}$.
    \end{enumerate}
    \item \cite[Theorem 1.1]{Homma} If $d = 3$ and $p = 2$, then the following are equivalent:
    \begin{enumerate}
    \item $X$ is not $F$-split.
    \item $X$ is isomorphic to the Fermat cubic surface.
    \item All smooth members of $|-K_X|$ are supersingular.
    \end{enumerate}
    \item \cite{Saito} If $d = 2$ and $p = 3$, then the following are equivalent:
    \begin{enumerate}
    \item $X$ is not $F$-split.
    \item $X$ is isomorphic to a double cover of $\mathbb{P}^2$ branched over the Fermat quartic curve.
    \item All smooth members of $|-K_X|$ are supersingular.
    \end{enumerate}
     \item \cite{Saito} If $d = 2$ and $p = 2$, then the following are equivalent:
     \begin{enumerate}
    \item $X$ is not $F$-split.
    \item $X$ is isomorphic to a double cover of $\mathbb{P}^2$ branched over a double line.
    \item All smooth members of $|-K_X|$ are supersingular.
     \end{enumerate}
\end{enumerate}
\end{theorem}

In the cases $(d,p) \in \{(3,2),(2,3)\}$, Homma and Saito also give a description of $X$ as a blow-up of $\mathbb{P}^2$. The goal of this article is to finish the classification of non-$F$-split del Pezzo surfaces by treating the case of degree $1$. 
Recall that in this case, the anti-canonical system $|-K_X|$ is a generically smooth pencil that induces an elliptic fibration on the blow-up of $X$ in the unique base point of $|-K_X|$. We let $\Delta \subseteq |-K_X| \cong \mathbb{P}^1$ be the vanishing locus of the discriminant of the induced fibration, considered as a divisor of degree $12$ (with multiplicities) on $\mathbb{P}^1$. With this terminology, we will prove the following:

\begin{theorem} \label{thm: main theorem}
Let $X$ be a del Pezzo surface of degree $1$ over an algebraically closed field $k$ of characteristic $p > 0$. Then, the following hold:
\begin{enumerate}
\item If $p = 5$, then the following are equivalent:
\begin{enumerate}
\item $X$ is not $F$-split.
\item $X$ is isomorphic to a double cover of $\mathbb{P}(1,1,2)$ branched over a sextic given by the equation $$x^3 + s^5t - st^5.$$
\item All smooth members of $|-K_X|$ are supersingular and $\Delta$ is projectively equivalent to $2\mathbb{P}^1(\mathbb{F}_5)$.
\end{enumerate}
\item If $p = 3$, then the following are equivalent:
\begin{enumerate}
\item $X$ is not $F$-split.
\item $X$ is isomorphic to a double cover of $\mathbb{P}(1,1,2)$ branched over a sextic given by an equation of one of the following forms:
\begin{eqnarray*} \quad \quad
x^3 + s^4x + a_6, \quad\quad
x^3 + s^3tx + a_6, \quad\quad
x^3 + (s^3t - st^3)x + a_6.
\end{eqnarray*}
\item All smooth members of $|-K_X|$ are supersingular and $\Delta$ is projectively equivalent to $12 [0:1]$, $9 [0:1] + 3 [1:0]$, or $3\mathbb{P}^1(\mathbb{F}_3)$.
\end{enumerate}
\item If $p = 2$, then the following are equivalent:
\begin{enumerate}
\item $X$ is not $F$-split.
\item $X$ is isomorphic to a separable double cover of $\mathbb{P}(1,1,2)$ branched over a cubic given by an equation of one of the following forms:
\begin{eqnarray*}
s^3, \quad\quad s^2t, \quad \quad s^2t + st^2
\end{eqnarray*}
\item All smooth members of $|-K_X|$ are supersingular.
\end{enumerate}
\end{enumerate}
\end{theorem}

\begin{remark}
Assuming that all smooth members of $|-K_X|$ are supersingular is not enough to guarantee that $X$ is not $F$-split in characteristic different from $2$.

For example, if $p = 5$, then every smooth surface given by an equation of the form
$$
y^2 = x^3 + a_6(s,t)
$$
satisfies that the generic member of $|-K_X|$ is supersingular, but $\Delta = V(a_6^2)$ is not projectively equivalent to $2\mathbb{P}^1(\mathbb{F}_5) \subseteq \mathbb{P}^1$ for a general choice of $a_6$. 

Similarly, if $p = 3$, then every smooth surface given by an equation of the form
$$
y^2 = x^3 + a_4(s,t)x + a_6(s,t)
$$
has generic member of $|-K_X|$ supersingular, but $\Delta = V(a_4^3)$ is not projectively equivalent to $3\mathbb{P}^1(\mathbb{F}_3)$ in general, even if $a_4$ has four distinct roots.
\end{remark}

\begin{remark}
There is a very recent notion of $n$-quasi-$F$-splitting studied in \cite{Takamatsu}, where the case $n = 1$ corresponds to the classical notion of $F$-splitting. By \cite[Corollary 4.7]{Takamatsu}, all smooth del Pezzo surfaces are $2$-quasi-$F$-split. In particular, the non-$1$-quasi-$F$-split del Pezzo surfaces appearing in Theorem \ref{thm: known} and Theorem \ref{thm: main theorem} are nevertheless $2$-quasi-$F$-split.
\end{remark}

\begin{remark}
In $(2)(b)$ of Theorem \ref{thm: main theorem}, the substitutions that preserve this type of equation are of the form $x \mapsto x + b_2$ with $b_2$ homogeneous of degree $2$ in $s$ and $t$. In $(3)(b)$ of Theorem \ref{thm: main theorem}, they are of the form $y \mapsto y + b_1x + b_3, x \mapsto b_2$ with $b_i$ homogeneous of degree $i$.
Since del Pezzo surfaces of degree $1$ have finite automorphism group, we can use this (and \cite[Theorem 3.4]{DolgachevMartin2}) to determine the dimensions of the moduli spaces of non-$F$-split del Pezzo surfaces:

\begin{table}[h!]
\begin{center}
\begin{tabular}{|l||*{3}{c|}}\hline
\backslashbox{Characteristic}{Degree}
&\makebox[3em]{$3$}&\makebox[3em]{$2$}&\makebox[3em]{$1$}
\\\hline\hline
$5$ & $-$& $-$ & $0$\\\hline
$3$ & $-$ & $0$& $4$\\ \hline
$2$ & $0$ & $3$& $6$\\\hline
\end{tabular}
\end{center}
\caption{Dimension of the moduli space of non-$F$-split del Pezzo surfaces}
\end{table}
\end{remark}

\bigskip

\begin{remark}
The unique non-$F$-split del Pezzo surface $X$ of degree $1$ in characteristic $5$ is also geometrically interesting from other perspectives. For example, it contains $144$ generalized Eckardt points, that is, points lying on ten $(-1)$-curves, and these points do not lie on the ramification curve of the anti-bicanonical map, answering \cite[Question 7.3]{Winter}:

To see this, note first that we can write the surface as
$$
y^2 = x^3 + s^6 + t^6
$$
by Lemma \ref{lem: LA5}. By \cite[Proposition 5.1]{Alvarado}, all $(-1)$-curves on $X$ are defined over $\mathbb{F}_{25}$. The associated rational elliptic surface has six cuspidal fibers $F_1,\hdots,F_6$ defined over $\mathbb{F}_{25}$. Each of the $240$ $(-1)$-curves on $X$ must pass through one of the $24$ smooth $\mathbb{F}_{25}$-points on $F_i$ that are different from the base point of $|-K_X|$ and the cusp. By \cite{DolgachevMartin}, the automorphism group of $X$ acts transitively on these $24$ points, hence there are precisely ten $(-1)$-curves through each of them. We conclude that there are $144$ generalized Eckardt points on $X$. By construction, none of these Eckardt points lie on the ramification curve of $X \to \mathbb{P}(1,1,2)$, since this curve meets the fibers $F_i$ precisely in the cusp.

This argument also shows that $X$ can be realized as the blow-up of $\mathbb{P}^2$ in eight $\mathbb{F}_{25}$-points in general position. We note that, over $\mathbb{F}_5$, there exist no eight points in general position, so this is a rather restrictive condition.
\end{remark}

To prove Theorem \ref{thm: main theorem}, we use an elementary computational approach: Applying Fedder's criterion (see Section \ref{sec: Fedder}) to the anti-canonical model of $X$, we characterize the non-$F$-splitting of $X$ in terms of conditions on the parameters appearing in the equation of the anti-canonical model. Then, we translate these conditions to the conditions of type (b) and (c) that appear in Theorem \ref{thm: main theorem}. By using this approach, we will also obtain a classification of non-$F$-split Gorenstein log del Pezzo surfaces of degree $1$ in characteristics $5$, $3$, and $2$ that the interested reader can extract from Lemmas \ref{lem: char5}, \ref{lem: char3}, and \ref{lem: char2}.

\begin{acknowledgments*}
We gratefully acknowledge the hospitality and support of the Max-Planck-Institut f\"ur Mathematik Bonn, where we started working on this article during an internship program. We thank Fabio Bernasconi for helpful comments on a first version of this article.
\end{acknowledgments*}

\section{Fedder's criterion for del Pezzo surfaces of degree $1$} \label{sec: Fedder}

To determine whether or not a del Pezzo surface $X$ of degree $1$ is $F$-split, we want to apply the following criterion (see \cite[Theorem 1.3]{Saito} and \cite[Proposition 3.1]{Smith}):

\begin{theorem}
Let $X$ be a projective variety over a perfect field $k$. Then, the following are equivalent:
\begin{enumerate}
\item $X$ is $F$-split.
\item The section ring $S_{\mathcal{L}} = \bigoplus_{n \in \mathbb{N}} H^0(X,\mathcal{L})$ is $F$-split for all invertible sheaves $\mathcal{L}$ on $X$.
\item The section ring $S_{\mathcal{L}} = \bigoplus_{n \in \mathbb{N}} H^0(X,\mathcal{L})$ is $F$-split for some ample invertible sheaf $\mathcal{L}$ on $X$.
\end{enumerate}
\end{theorem}

Here, a $k$-algebra $R$ is called \emph{$F$-split} if the map $F:R \to F_*R$ splits as a map of $R$-modules. In the del Pezzo case, we have a natural choice for an ample invertible sheaf, namely $\mathcal{L} = \omega_X^{-1}$. 
Recall that the anti-canonical model of a del Pezzo surface $X$ of degree $1$ is a sextic hypersurface in $\mathbb{P}(1,1,2,3)$ given by an equation $f$ of the form
\begin{equation} \label{eq: dP1 general}
f = y^2 + a_1(s,t)xy + a_3(s,t)y - (x^3 + a_2(s,t)x^2 + a_4(s,t)x + a_6(s,t)),
\end{equation}
where $a_i$ is a homogeneous polynomial of degree $i$ in the weight $1$ variables $s$ and $t$. In other words, we have
$$
S_{\omega_X^{-1}} \cong k[s,t,x,y]/(f).
$$
To determine the $F$-splitting of this ring, we use Fedder's criterion \cite[Theorem 1.12]{Fedder}:

\begin{lemma} \label{lem: Fedder}
Let $k$ be a perfect field and $f \in k[x_1,\hdots,x_n]$. Then, the following are equivalent:
\begin{enumerate}
\item The ring $k[x_1,\hdots,x_n]/(f)$ is $F$-split.
\item We have $f^{p-1} \not \in (x_0^p,\hdots,x_n^p)$.
\end{enumerate}
\end{lemma}

\begin{example}
Let $k$ be an algebraically closed field of characteristic $p > 0$. If $X$ is an Abelian variety over $k$, then $X$ is $F$-split if and only if it is ordinary (see e.g. \cite[Lemma 1.1]{MehtaSrinivas}).

In particular, if $X$ is an elliptic curve over $k$, then $X$ is not $F$-split if and only if it is supersingular. If $p \in \{2,3,5\}$, it is well-known (see e.g. \cite[Theorem 4.1, Exercise 5.7]{Silverman}), that $X$ is supersingular if and only if its $j$-invariant is $0$. This can also be checked by applying Fedder's criterion to a Weierstraß equation of $X$.
\end{example}

To simplify the equations we get, we will use the following observation from linear algebra:

\begin{lemma} \label{lem: LA5}
Let $k$ be an algebraically closed field of characteristic $p > 0$. Let $a,b,c,d \in k$ be such that the homogeneous polynomial $g = as^{p+1} + bs^pt + cst^p + dt^{p+1}$ is non-zero and has at least $3$ distinct roots on $\mathbb{P}^1$. Then, there is a linear change of coordinates $\sigma \in {\rm GL}_2(k)$ such that $\sigma(g) = s^pt - st^p$.
\end{lemma}

\begin{proof}
The key observation is that a linear change of coordinates sends $g$ to an equation of the same form, possibly with different $a,b,c,d \in k$. After such a change of coordinates, we may assume that $(1,0),(0,1)$, and $(1,1)$ are roots of $g$. This is equivalent to saying that $a = d = 0$ and $b = -c$. Thus, after rescaling, $g$ has the desired form.
\end{proof}

\section{Characteristic $5$}
In this section, we prove Theorem \ref{thm: main theorem} $(1)$.
We have $p = 5$, so we can complete the squares and cubes in Equation \eqref{eq: dP1 general} to assume that the del Pezzo surface $X$ is given by an equation of the form
\begin{equation}\label{eq: dp1 char 5}
f = y^2 - (x^3 + a_4(s,t)x + a_6(s,t)).
\end{equation}

\begin{lemma} \label{lem: char5}
Let $X \subseteq \mathbb{P}(1,1,2,3)$ be a surface over an algebraically closed field $k$ of characteristic $p = 5$ given by Equation \eqref{eq: dp1 char 5}. Then, $X$ is not $F$-split if and only if $a_4 = 0$ and $a_6 = as^6 + bs^5t + cst^5 + dt^6$ for some $a,b,c,d \in k$.
\end{lemma}

\begin{proof}
Using Lemma \ref{lem: Fedder}, we know that $X$ is not $F$-split if and only if $f^4 \in (x^5,y^5,s^5,t^5) =: I$.
We observe that the only products of four of the summands that appear in $f$ that are not automatically in $I$ are the following:
$$
a_4x^4y^4, a_4^2x^2y^4, a_6x^3y^4
$$
Now, $a_4$ is homogeneous of degree $4$, so $a_4 x^4y^4 \in I$ if and only if $a_4=0$.

Similarly, $a_6$ is homogeneous of degree $6$, so $a_6y^4x^3 \in I$ if and only if $a_6$ is of the stated form.
\end{proof}

A surface $X$ given by Equation \eqref{eq: dp1 char 5} is the Weierstra{\ss} equation of a rational elliptic surface over $\mathbb{P}^1$. This elliptic fibration has the following $j$-invariant and discriminant
\begin{eqnarray*}
\Delta &=&  a_4^3 - 2 a_6^2\\
j &=& \frac{3a_4^3}{\Delta}.
\end{eqnarray*}
So, asking that $a_4 = 0$ is the same as asking that $j = 0$, which, in characteristic $5$, is the same as asking that all smooth members of $|-K_X|$ are supersingular. Once $a_4 = 0$, it follows easily from the Jacobian criterion that $X$ is smooth if and only if $\Delta$ has only double roots if and only if $a_6$ has $6$ distinct roots. 
Combining Lemma \ref{lem: char5} and Lemma \ref{lem: LA5}, we obtain Theorem \ref{thm: main theorem} (1), since the roots of $s^5t-st^5$ are precisely the $\mathbb{F}_5$-rational points of $\mathbb{P}^1$. 



\section{Characteristic $3$}

In this section, we prove Theorem \ref{thm: main theorem} $(2)$.
We have $p = 3$, so we can complete the squares but not the cubes in Equation \eqref{eq: dP1 general} to assume that the del Pezzo surface $X$ is given by an equation of the form
\begin{equation}\label{eq: dp1 char 3}
f = y^2 - (x^3 + a_2(s,t)x^2 + a_4(s,t)x + a_6(s,t)).
\end{equation}

\begin{lemma} \label{lem: char3}
Let $X \subseteq \mathbb{P}(1,1,2,3)$ be a surface over an algebraically closed field $k$ of characteristic $p = 3$ given by Equation \eqref{eq: dp1 char 3}. Then, $X$ is not $F$-split if and only if $a_2 = 0$ and $a_4 = as^4 + bs^3t + cst^3 + dt^4$ for some $a,b,c,d \in k$.
\end{lemma}

\begin{proof}
Using Lemma \ref{lem: Fedder}, we know that $X$ is not $F$-split if and only if $f^2 \in (x^3,y^3,s^3,t^3) =: I$.
We observe that the only products of two of the summands that appear in $f^2$ that are not automatically in $I$ are the following:
$$
a_2 x^2y^2, a_4 xy^2
$$
Since $a_2$ is homogeneous of degree $2$ in $s$ and $t$, we have $a_2x^2y^2 \in I$ if and only if $a_2 =0$.

Similarly, $a_4$ is homogeneous of degree $4$ in $s$ and $t$, so $a_4 xy^2\in I$ if and only if $a_4$ is of the stated form.
\end{proof}

The genus $1$ fibration associated to a surface $X$ with Equation \eqref{eq: dp1 char 3} has the following $j$-invariant and discriminant
\begin{eqnarray*}
\Delta &=& -a_2^2(a_2a_6 - a_4^2) - a_4^3\\
j &=& \frac{a_2^6}{\Delta}.
\end{eqnarray*}
If $X$ is smooth, then the fibration is elliptic, hence $\Delta \neq 0$, so asking that $a_2 = 0$ is the same asking that all smooth members of $|-K_X|$ are
supersingular. Combining Lemma \ref{lem: char3} and Lemma \ref{lem: LA5}, we obtain Theorem \ref{thm: main theorem} (2), with the three cases corresponding to $\Delta$ having $1$, $2$, or at least $3$ roots.



\section{Characteristic $2$}
In this section, we prove Theorem \ref{thm: main theorem} $(3)$.

\begin{lemma} \label{lem: char2}
Let $X \subseteq \mathbb{P}(1,1,2,3)$ be a surface over an algebraically closed field $k$ of characteristic $p = 2$ given by Equation \eqref{eq: dP1 general}. Then, $X$ is not $F$-split if and only if $a_1 = 0$.
\end{lemma}

\begin{proof}
Using Lemma \ref{lem: Fedder}, we know that $X$ is not $F$-split if and only if $f \in (x^2,y^2,s^2,t^2) =: I$.
Since the $a_i$ are homogeneous of degree $i$ in $s$ and $t$, the only summand of $f$ that is not automatically in $I$ is $a_1xy$ and we have $a_1xy \in I$ if and only if $a_1 = 0$.
\end{proof}

The genus $1$ fibration associated to a surface $X$ with Equation \eqref{eq: dP1 general} has the following $j$-invariant and discriminant
\begin{eqnarray*}
\Delta &=& a_1^4(a_1^2a_6 + a_1a_3a_4 + a_2a_3^2 + a_4^2) + a_1^3a_3^3 + a_3^4\\
j &=& \frac{a_1^{12}}{\Delta}.
\end{eqnarray*}
If $X$ is smooth, then the fibration is elliptic, hence $\Delta \neq 0$, so asking that $a_1 = 0$ is the same asking that all smooth members of $|-K_X|$ are
supersingular.

Now, in characteristic $2$, the curve given by $a_1x - a_3$ is the branch locus of the double cover $X \to \mathbb{P}(1,1,2)$. This curve is a union of three (not necessarily distinct) lines through the vertex if and only if $a_1 = 0$. This proves Theorem \ref{thm: main theorem} (3).


\bibliographystyle{amsplain}
\bibliography{dP}
\end{document}